%% file: idlaSG.tex
\documentclass[11pt]{scrartcl}

\usepackage[utf8]{inputenc}
\usepackage{cite}
\usepackage{amsmath,amsthm,amssymb}
\usepackage{dsfont}
\usepackage{mathtools}

\usepackage{tikz}
\usepackage{pgfmath}
\usepackage{pgfplots}

\usetikzlibrary{fadings}

\usepackage{float}

\newcommand{\R}{\mathbb{R}}
\newcommand{\E}{\mathbb{E}}
\newcommand{\Z}{\mathbb{Z}}
\newcommand{\N}{\mathbb{N}}
\newcommand{\Prob}{\mathbb{P}}

\newtheorem{thm}{Theorem}[section]
\newtheorem{lemma}{Lemma}[section]

\begin{document}

{\centerline{ \LARGE{On the fluctuations of Internal DLA on the Sierpinski gasket graph}}}~\\[-2mm]

\centerline{\large{Nico Heizmann}}
\centerline{\textit{ Technische Universität Chemnitz}} ~\\

%
%

\begin{abstract}
	Internal diffusion limited aggregation (IDLA) is a random aggregation model on a graph $G$, whose clusters are formed by random walks started in the origin (some fixed vertex) and stopped upon visiting a previously unvisited site. On the Sierpinski gasket graph the asymptotic shape is known to be a ball in the usual graph metric. In this paper we establish bounds for the fluctuations of the cluster from its asymptotic shape.
\end{abstract}

\section{Introduction}
The \textit{internal diffusion limited aggregation model} (\textit{IDLA}), which was introduced by Diaconis and Fulton in \cite{growthAlgebra}, is a stochastic aggregation model growing by consecutively started particles, where a particle is added upon exiting the current cluster.
Let $G$ be an infinite but locally finite connected graph with a specified vertex $\circ$ acting as the origin of these particles. Let $\big(X^1(t)\big)_{t\geq 0}, \big(X^2(t)\big)_{t\geq 0}, \ldots$ be a sequence of iid simple random walks on $G$ started in $\circ$ representing the particles. The IDLA cluster $\mathcal{I}(i)$ after $i\geq 1$ particles is now iteratively defined as 
\begin{align*}
	&\mathcal{I}(0) = \emptyset, &\mathcal{I}(i) = \mathcal{I}(i-1) \cup \{ X^i(\sigma^i) \}
\end{align*}
where $\sigma^i = \inf\{ t\geq 0 | X^i(t) \notin \mathcal{I}(i-1) \}$ is the time particle $i$ hits the outer boundary of the already existing cluster.

A typical question concerning IDLA is the typical shape of the random set $\mathcal{I}(i)$ for large $i$: On $\Z^d$ Lawler, Bramson and Griffeath \cite{lawler1992} identified the limit shape as a euclidean ball and later in \cite{lawlerSubdiffusive} Lawler improved the bounds for the fluctuations. Finally, later in \cite{AsselahD3} and in \cite{jerisonD3} these bounds were further improved to sublogarithmic bounds for $d\geq 3$ and logarithmic ones for $d=2$ in \cite{Asselah_2013} and \cite{jerisonD2} via two different approaches.

In this work we will use the ideas from \cite{lawlerSubdiffusive} to improve the fluctuation bounds for the Sierpinski gasket graph, whose limit shape is already identified as a ball in the graph metric, see \cite{idlaSG}.
The main improvement from the already existing result is based on the analysis of the sandpile model, which is a different aggregation model. We will use the exact calculation of the odometer function of the sandpile model from \cite{HSH17} to develop an improved inner bound and the outer bound  then follows analogously from \cite{idlaSG}.

Since it avoids unnecessary technicalities, we work on the double sided Sierpinski gasket graph $SG$ as this gives us a regular graph, whose vertices share degree $4$.
\begin{figure}[h]
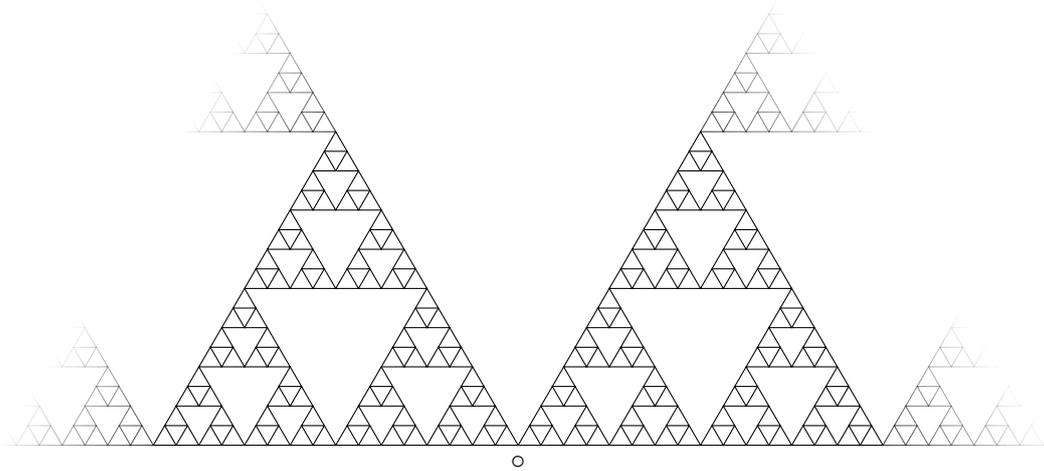

	\include{tikzSG}
\caption{The doubled Sierpinski gasket graph}
\end{figure}
For the construction of $SG$ let
\begin{align*}
	&V_0 := \Big\{ (0,0), (1,0), (1/2,\sqrt{3}/2) \Big\}, \\
	&E_0 := \Big\{ \big((0,0),(1,0)\big),\big((1,0),(1/2,\sqrt{3}/2)\big),\big((0,0),(1/2,\sqrt{3}/2)\big) \Big\}
\end{align*}
and then recursively 
\begin{align*}
	V_{n+1} &= V_n \cup \bigg((2^n,0)+V_n \bigg) \cup \bigg((2^{n-1}, 2^{n-1} \sqrt{3}) + V_n \bigg) \\
	E_{n+1} &= E_n \cup \bigg((2^n,0)+E_n \bigg) \cup \bigg((2^{n-1}, 2^{n-1} \sqrt{3}) + E_n \bigg)
\end{align*}

where $(x,y)+S = \{(x,y)+z | z\in S\}$. The Sierpinski gasket graph is defined by $(V_\infty = \bigcup_{n\in\N} V_n, E_\infty = \bigcup_{n\in\N} E_n)$ and from this we get the double sided Sierpinski gasket graph $SG$ by adding an at the y-axis reflected copy $SG = (-V_\infty \cup V_\infty, -E_\infty \cup E_\infty)$.

Set $\circ=(0,0)$ as the point the cluster grows around and denote by $B_x(n)$ the closed balls of radius $n$ around vertex $x$ in the usual graph metric and the volume of balls around $\circ$ as $b_n:=|B_\circ(n)|$. 

In \cite{idlaSG} Chen, Huss, Sava-Huss and Teplayev were able to proof a basic shape result on $SG$.
\newpage
\begin{thm}[{\cite[Theorem 1.1]{idlaSG}}]
	On $SG$ the IDLA cluster after $b_n=|B_\circ(n)|$ particles occupies a set of sites close to ball of radius $n$. That is, for all $\varepsilon>0$, we have
	\begin{align*}
		B_\circ(n(1-\varepsilon)) \subseteq \mathcal{I}(b_n) \subseteq B_\circ(n(1+\varepsilon)) \text{ with probability }1.
	\end{align*}
\end{thm}
Our main result improves the linear bounds from \cite{idlaSG} to sublinear bounds similar to the ones in \cite{lawlerSubdiffusive}.

\begin{thm}\label{mainThm}
	On $SG$ the IDLA cluster after $b_n=|B_\circ(n)|$ particles satisfies with probability 1 for any $\kappa>0$, some constant $c>0$ and $n$ large enough
	\begin{align*}
		B_\circ(n-c n^\frac{1}{2}\ln(n)^\frac{1+\kappa}{2\alpha}) \subseteq \mathcal{I}(b_n) \subseteq B_\circ(n+c\ n^{\frac{1}{2}+\frac{1}{2\alpha}}\ln(n)^{(1-\frac{1}{\alpha})\frac{1+\kappa}{2\alpha}}),
	\end{align*}
	where $\alpha=\frac{\log(3)}{\log(2)}$ is the Hausdorff dimension of the Sierpinski gasket.
\end{thm}


\section{Preliminaries}
Let $G=(V,E)$ be an infinite locally finite connected graph. For convenience we will write $x\in G$ instead of $x\in V$ and $x\sim y$ if $(x,y)\in E$. Furthermore, we write $d(x,y)$ for the usual graph distance between vertices $x,y\in G$ and for $x\in G$ and $A\subseteq G$ denote
\begin{align*}
	\deg(x) &:= \{z \in G \ | \ x\sim z \} \text{ the vertex degree of $x$, } \\
	B_x(n) &:= \{z\in G\ |\ d(x,z) \leq n \} \text{ the ball of radius $n$ with center $x$, }\\
	\partial_I A &:= \{y\in A\ |\ \exists_{z\notin A}\ y\sim z \} \text{ the inner boundary of $A$}.
\end{align*}

Let $\Prob_x$ be the probability law of the simple random walk $X= (X(t))_{t\geq 0}$ on $G$ started in vertex $x\in G$ and $\E_x$ the expectation with respect to $\Prob_x$. For this simple random walk we define the stopping times
\begin{align*}
	&\tau_x(n) := \inf\{ t\geq 0 \ |\ X(t) \in \partial_I B_x(n) \cup B_x^c(n)\}, &\tau(n) := \tau_\circ(n)
\end{align*}
and the stopped Green function
\begin{align*}
	g_n(x,y) := \E_x\bigg(\sum_{n=0}^{\tau(n)-1} \mathds{1}_{\{X(t)=y\}} \bigg)
\end{align*}
which plays a major role in the analysis of bounds of IDLA. For any function $f:G\rightarrow\R$ we define the (probabilistic) graph Laplacian as 
\begin{align*}
	\Delta f (x) := \frac{1}{\deg(x)}\sum_{y\sim x} h(y)-h(x)
\end{align*}
and call any $h:G\rightarrow \R$ satisfying $\forall_{x\in S}\ \Delta h(x) = 0$ a harmonic function on $S\subseteq G$.
It is easy to show that the prementioned stopped Green function satisfies $\Delta g_n(\circ,x) = \delta_\circ(x)$ for all $x\in B_\circ(n)\backslash\partial_I B_\circ(n)$ and therefore is harmonic on $B_\circ(n)\backslash(\{\circ\}\cup\partial_I B_\circ(n))$.

Now, due to its special structure, one can derive three major properties of $SG$, all of which can be proven using coverings of proper triangles.
\newpage
\begin{lemma}
	The $SG$ satisfies the following properties.
	\begin{enumerate}
		\item[(EHI)] The elliptic Harnack inequality: there exists a positive constant $C$ such that for all $x\in G$, $n>0$ and functions $h\geq 0$ being harmonic on $B_x(2n)$ it holds
		\begin{align*}
			\sup_{y\in B_x(n)} h(y) \leq C \inf_{y\in B_x(n)} h(y).
		\end{align*} 
		\item[$(V_\alpha)$] The uniform volume growth condition: there exist constants $c,C>0$ such that for all $x\in G$, $n>0$ it holds
		\begin{align*}
			c n^\alpha \leq |B_x(n)| \leq C n^\alpha.
		\end{align*}
		\item[$(E_\beta)$] The uniform exit time growth condition: there exist constants $c,C>0$ such that for all $x\in G$, $n>0$ it holds
		\begin{align*}
			c n^\beta \leq \E_x(\tau_x(n)) \leq C n^\beta.
		\end{align*}
	\end{enumerate}
	The respective constants in the exponents are  $\alpha := {\log(3)}/{\log(2)}$ the Hausdorff dimension and $\beta:={\log(5)}/{\log(2)}$ the walk dimension of the Sierpinski gasket.
\end{lemma}

\begin{proof}
	See \cite[Chapter 2.2]{HeatKernelFractal} for the proofs of $(V_\alpha)$ and $(E_\beta)$. The proof of $(EHI)$ is then given in \cite[Theorem 2.6]{HeatKernelFractal}
\end{proof}

For our proof of the inner bound in Theorem \ref{mainThm} we need some slightly sharper bounds than the ones used in \cite[Lemma 2.8 \& 2.10]{idlaSG} whose proofs actually already give the desired result:

\begin{lemma}\label{hilfsLemma}
	There exist constants $c_1, c_2>0$ such that for every $x\in SG$ and $n>0$ we have
	\begin{align*}
		\E_x(\tau_\circ(n)) &\geq c_1\ d(x,\partial_I B_\circ(n))^\beta \text{ and } \\
		g_n(x,x) &\leq c_2\ d(x,\partial_I B_\circ(n))^{\beta-\alpha}.
	\end{align*}
\end{lemma}

In the remainder of this paper constants may change from line to line.

\section{The divisible sandpile}
In this section we will establish a lower bound in Lemma \ref{lemmaOdometerBound} needed for our arguments in Section \ref{sectionInnerBound}. We derive this from the results in \cite{HSH17}, which are briefly outlined here.

We call a function $\mu:SG\rightarrow[0,\infty)$ with finite support $|\text{supp}(\mu)|<\infty$ a sand distribution and we call any $x\in SG$ unstable if $\mu(x)>1$. Any unstable vertex can be toppled, such that the excess mass $\mu(x)-1$ is split evenly among its neighboring vertices: The resulting distribution is then given by
\begin{align*}
	T_x \mu := \mu + \max\{\mu(x)-1,0\} \Delta \delta_x,
\end{align*}
where $\delta_x(y)$ equals $1$ if $x=y$ and $0$ otherwise. We call $T_x$ the toppling operator at vertex $x$. Note that, there is no need for $x$ to be an unstable vertex since otherwise $T_x = \mathrm{Id}$. We start with an initial sand distribution $\mu_0$ and let $(x_k)_{k\geq 1}$ be a sequence of vertices containing each vertex of $SG$ infinitely often. We call such sequences toppling sequences and define the sand distribution after $k$ topples as
\begin{align*}
	\mu_k(y) := T_{x_k} \mu_{k-1}(y) = T_{x_k} \ldots T_{x_1} \mu_0(y)
\end{align*}
as well as the odometer function $u_k$, which counts the mass emitted by a vertex up to $k$ topples 
\begin{align*}
	u_k(y) := \sum_{j\in \{i\leq k | x_i = y\} } \mu_j(y)-\mu_{j+1}(y) = \sum_{j\in \{i\leq k | x_i = y\} } \max\{\mu_j(y)-1,0\}.
\end{align*}
Intuitively, toppling many vertices should spread the mass out such that the mass is covered by more and more vertices until there is not enough mass left to cover any new vertices. This intuition turns out to be the case as the next lemma states.
 
\begin{lemma}[{\cite[Lemma 3.1]{SandpileZ}}]\label{sandpileDirichletLemma}
As $k\rightarrow\infty$, $\mu_k$ converges to a sandpile distribution $\mu$ and $u_k$ converges from below to a limit function $u$. Moreover, these limits satisfy
	\begin{align*}
		\mu(z) = \mu_0(z) + \Delta u(z) \text{ and } \mu(z) \leq 1, \text{ for any } z\in SG.
	\end{align*}
\end{lemma}

At first sight the limits $u$ and $\mu$ seem to depend on the choice of the toppling sequence selected for the toppling procedure. This however, turns out not to be the case as the Abelian property of the divisible sandpile model states.

\begin{lemma}[{\cite[Lemma 3.6]{HSH17}}]
	The limiting odometer function $u$ is independent of the choice of the toppling sequence.
\end{lemma}

Therefore we call $\mu$ the sand distribution and $u$ the odometer function according to the starting distribution $\mu_0$ and its sandpile cluster is defined as $\mathcal{S}:=\{ z \in SG \ | \ \mu(z) = 1 \} $. In our case we are particularly interested in the limiting functions according to the starting distribution $\mu_0 = |B_\circ(n)| \delta_\circ$ since the odometer function then satisfies $\Delta u = 1- |B_\circ(n)| \delta_\circ$ on the sandpile cluster. This will help us in the analysis of the stopped Green function $g_n$.
The next lemma gives the solution to this problem and is a direct consequence of a result from Huss and Sava-Huss in \cite[Theorem 4.2]{HSH17}.
\begin{lemma}\label{sandpileClusterLemma}
	For any $n\geq 1$ the sandpile distribution and therefore the sandpile cluster according to the starting distribution $\mu_0 = |B_\circ(n)| \delta_\circ$ on $SG$ are given by
	\begin{align*}
		&\mu(z) = \mathds{1}_{B_\circ(n)}(z) = \begin{cases}
			1 \text{ for } z \in B_\circ(n) \\
			0 \text{ otherwise}
		\end{cases} \text{ and } 
		&\mathcal{S} = B_\circ(n).
	\end{align*}
\end{lemma}
Note that, from this we also know that the odometer function on $z\in \partial_I B_\circ(n)$ equals $u(z)=0$ since otherwise there would be mass outside of the cluster.
For our analysis we will need a lower bound for the odometer function depending on the distance to the boundary $\partial_I B_\circ(n)$, which we will derive from the calculations in \cite[Section 5]{HSH17}.

Let $\tilde{u}:SG^+\rightarrow\R$ be the function on $SG^+=SG\cap(\R_{\geq0} \times \R) = (V_\infty,E_\infty)$ defined by 
\begin{align*}
	&\tilde{u}(x,0) = 0 \text{ for all } x\geq 0,\\
	&\tilde{u}(x,\sqrt{3}/2) = 2, \text{ for all } x\geq 0,\text{ s.t. } (x,\sqrt{3}/2)\in SG^+ 
\end{align*}
and $\Delta \tilde{u}(x,y) = 1$ for all $(x,y)\in SG^+$.

\begin{lemma}[{\cite[Lemma 5.11]{HSH17}}]
	Let $k\in\N$, then for $z=(2^{k-1},2^{k-1} \sqrt{3} )$ the upper boundary point of $V_n$, which is a triangle of size $2^k$, it holds $\tilde{u}(z) = 2\cdot 5^k$. 
\end{lemma}

Let $\psi_k:\R^2 \rightarrow \R^2$ be given by $\psi_k(x,y) = (-\frac{1}{2} x - \frac{\sqrt{3}}{2} y+2^k,\frac{\sqrt{3}}{2} x- \frac{1}{2} y)$, which rotates $V_k$ by $120^\circ$ around its big center hole. With this and the function $\tilde{u}$ one can calculate the odometer function for specific starting masses.

\begin{thm}[{\cite[Theorem 5.12]{HSH17}}]\label{odometerTheo}
	Let $u^{(k)}:SG\rightarrow\R$ be the odometer function of the divisible sandpile with initial mass distribution $\mu_0 = 3^{k+1} \delta_\circ = (|B_\circ(2^k)|-2) \delta_\circ$. Then for all $k\geq 0$
	\begin{align*}
		u^{(k)}(x,y) = \begin{cases}
			\tilde{u} \circ \psi_k (|x|,y) \text{ if } (x,y)\in B_\circ(2^k), \\
			0 \text{ otherwise}.
		\end{cases}
	\end{align*}
\end{thm}

Notice that together with the previous lemma we get $u^{(k)}(\circ) = 2\cdot 5^k$.
Using this and the fact the odometer grows when adding the missing mass of $2$, we can derive a lower bound for the odometer function:

\begin{lemma}\label{lemmaOdometerBound}
	Let $n,\delta \in \N$ such that $\delta \leq n/2 $ and $u:SG\to \R$ the odometer function of the sandpile cluster according to the starting distribution $\mu_0 = |B_\circ(n)| \delta_0$. Then for all $z\in B_\circ(n-3\delta)$ it holds
	\begin{align*}
		u(z) \geq c\ \delta^{\beta}
	\end{align*}
	for some $c>0$.
\end{lemma}
	
\begin{proof}
	Let $m\in\N$ such that $ 2^{m+1} > \delta \geq 2^m$. Then there are triangles $\triangle_i$ of size $2^{m-1}$ inside the annulus $B_\circ(n)\backslash B_\circ(n-\delta)$, whose removal leaves $\circ$ in a finite component of a union of larger triangles. Now let $z_i\in\triangle_i$ be the boundary point closest to $\circ$. For all such boundary points $z_i,z_j$ it holds $u(z_i)=u(z_j)$ due to symmetry.
	Now any lower bound for $u(z_i)$ also holds on $u(z)$ for all $z\in B_\circ(n-\delta)$, since the odometer is decreasing in distance to $\circ$. 
	Let $u_\triangle$ be the odometer when starting with just enough mass in $\circ$ to fill up all the triangles $\triangle_i$, then we obviously have $u_\triangle \leq u$ since after that we only add more mass into the system which has to be distributed to the outer boundary. Now $u_\triangle(z_i)$ equals the odometer at vertex $\circ$ of the sandpile cluster with starting mass $|\triangle_i|$ and from Theorem \ref{odometerTheo} we can deduce:
	\begin{align*}
		u(z_i) \geq u_\triangle(z_i) \geq u^{{(m-1)}}(\circ) = 2 \cdot 5^{m-1} \geq c\ \delta^\beta 
	\end{align*}
	for some $c>0$.
	
\end{proof}

\section{The inner bound for IDLA}\label{sectionInnerBound}
Similar to \cite{idlaSG} we use the standard approach for bounds of the IDLA cluster from \cite{lawlerSubdiffusive}, which heavily depends on the analysis of stopped Green functions.\newline 
We consider fixed $z\in B_\circ(n-c\ n^\frac{\alpha}{2}\ln(n)^\frac{1+\kappa}{2} )$ and for $i=1,\ldots,b_n$ let $(X^i(t))_{t\geq 0}$ be the random walks generating the IDLA cluster. Now let
	\begin{align*}
		M&:= |\{1\leq i\leq b_n\ | \ \tau^i_z < \tau^i(n)\} | = \sum_{i=1}^{b_n} \mathds{1}_{\tau^i_z < \tau^i(n) }\\
		L&:=\big|\{1\leq i\leq b_n\ |\ \sigma^i < \tau_z^i < \tau^i(n)\}\big| = \sum_{i=1}^{b_n} \mathds{1}_{\sigma^i < \tau^i_z < \tau^i(n)}
	\end{align*}
	where $\sigma^i:=\inf\{t\geq 0 \ |\ X^i(t) \notin \mathcal{I}(i-1)\}$ and $\tau^i_z := \inf\{t\geq 0\ |\ X^i(t)=z\}$. So $M$ equals the number of random walks visiting vertex $z$ before hitting the boundary $\partial_IB_\circ(n)$, whereas $L$ equals the number of those walks that visit $z$ after the according particle has already been added to the cluster. 
	Therefore we have: 
	$$\{z \notin \mathcal{I}(b_n)\} \subseteq \{M=L\}. $$ 
	Now for any $a\geq 0$ we have $$\Prob\big(z \notin \mathcal{I}(b_n)\big) \leq \Prob(M=L) \leq \Prob\big(\{M \leq a\} \cup \{L \geq a\}\big) \leq \Prob(M\leq a) + \Prob(L\geq a) \text{.}$$ We will look for a specific $a$ giving us bounds, that vanish fast enough. Now for $M$ the summands are obviously iid, whereas the summands of $L$ are not. Therefore we will instead consider $\tilde{L}$ satisfying $\Prob(\tilde{L} \geq a) \geq \Prob(L \geq a)$ for all $a\in \R$ as follows. 
	For every $y\in B_\circ(n)$ there is at most one $i$ for which $X^i({\sigma^i}) = y$, since $y$ is already inside the cluster for the following indices $j>i$. Additionally, after the time $\sigma^i$ the random walk $X^i$ has the same distribution as an in $y=X^i({\sigma^i})$  started random walk.
	With the random variable $\mathds{1}^y_{\tau_z<\tau(n)}$ denoting the indicator function of the event that a random walk started in $y$ visits vertex $z$ before hitting the boundary $\partial_IB_\circ(n)$ we have
	\begin{align*}
		L \stackrel{\mathcal{D}}{=} \sum_{i=0}^{b_n} \sum_{y\in B_{\circ}(n)} \mathds{1}_{\{X^i({\sigma^i}) = y\}} \mathds{1}^y_{\tau_z < \tau(n)} \leq \sum_{y\in B_\circ(n)} \mathds{1}^y_{\tau_z < \tau(n)} =:\tilde{L}.
	\end{align*}
	
	Now the summands in $\tilde{L}$ are independent and we can easily calculate
	\begin{align*}
		\E(\tilde{L}) &= \sum_{y\in B_\circ(n)} P_y(\tau_z < \tau(n)) 
			=  \frac{1}{g_n(z,z)}\sum_{y\in B_\circ(n)} g_n(y,z) = \frac{1}{g_n(z,z)} \E_z(\tau(n)),\\
		\E(M) &= b_n\ P_\circ(\tau_z < \tau(n)) =  b_n \frac{g_n(\circ,z)}{g_n(z,z)}. 
	\end{align*}
	
	With the help of the following lemma we are able to calculate $\E(M)-\E(\tilde{L})$:
	
\begin{lemma}
		Let $u$ be the odometer function of the divisible sandpile on $SG$ for the starting distribution $\mu_0 =|B_\circ(n)| \delta_0$, then it holds for all $x\in G$
		$${|B_\circ(n)|}\ g_n(\circ,x) - \sum_{y\in B_\circ(n)} g_n(y,x)  = u(x). $$
	\end{lemma}
	\begin{proof}
		By Lemma \ref{sandpileDirichletLemma} and Lemma \ref{sandpileClusterLemma} the odometer function solves the following Dirichlet problem.
		\begin{align*}
			\begin{cases}
				\Delta u(x) = \Big(1 - |B_\circ(n)|\delta_\circ(x) \Big), \text{ if } x\in B_\circ(n)\backslash\partial_I B_\circ(n) \\
				u(x) = 0 , \text{ if } x\in \partial_I B_\circ(n)
			\end{cases}
		\end{align*}
		Now let $h_n(z) =  {|B_\circ(n)|}\ g_n(\circ,z) - \sum_{y\in B_\circ(n)} g_n(y,z)$ then we have for $z\in B_\circ(n)\backslash\partial_I B_\circ(n)$
		\begin{align*}
			\Delta h_n (x) &= \sum _{y\in B_\circ(n)} \delta_y(x) - |B_\circ(n)| \delta_0(x) \\
			&= (1-|B_\circ(n)| \delta_\circ(x))
		\end{align*}
		and for $x\in\partial_I B_\circ(n)$, $y\in SG$ we have $g_n(y,x)=0$ by definition. So $h_n$ and $u$ solve the same Dirichlet problem and by the uniqueness principle we have $h_n = u$.
	\end{proof}
	Now since by Lemma \ref{hilfsLemma} $\E_z(\tau(n))\geq c_1\  d(z,\partial_I B_\circ(n))^\beta$, $g_n(z,z)\leq c_2\ d(z,\partial_I B_\circ(n))^{\beta-\alpha}$ we have the two following major inequalities
	\begin{align*}
		c\ n^\alpha \geq \E(M) \geq \E(\tilde{L})&\geq \frac{{c}}{g_n(z,z)} d(z,\partial_I B_\circ(n))^\beta \geq c\ d(z,\partial_I B_\circ(n))^\alpha, \\
		\E(M) - \E(\tilde{L}) &= \frac{1}{g_n(z,z)} \bigg( {|B_\circ(n)|}\ g_n(\circ,z) - \sum_{y\in B_\circ(n)} g_n(y,z) \bigg) \\
		&\geq c\ d(z,\partial_I B_\circ(n))^{-(\beta-\alpha)}\ u(z) \geq c\  d(z,\partial_I B_\circ(n))^\alpha
	\end{align*} 
	where in the last inequality we used the bound $u(z)\geq c d(z,\partial_I B_\circ(n))^\beta$ from Lemma \ref{lemmaOdometerBound}.
	
	\begin{lemma}[{\cite[Lemma 4]{lawler1992}}]
		Let $S$ be a finite sum of independent indicator functions and $\E(S)=\mu$. Then for all sufficiently large $n$ and $0<\gamma<\frac{1}{2}$
		\begin{align*}
			\Prob(|S-\mu|\geq \mu^{\frac{1}{2}+\gamma}) \leq 2 \exp\Big(-\frac{1}{4} \mu^{2\gamma}\Big).
		\end{align*}
	\end{lemma}
	Choose $\gamma,\gamma'$ such that $\E(M)^{2\gamma}=\ln(n)^{1+\kappa}$ and $\E(L)^{2\gamma'}=\ln(n)^{1+\kappa}$ then we have for $p_n := cn^{-\ln(n)}$ declining faster to $0$ than any polynomial:
	\begin{align*}
		p_n &\geq \Prob(|M-\E(M)|\geq \E(M)^{\frac{1}{2}}\ln(n)^\frac{1+\kappa}{2}) + \Prob(|L-\E(L)|\geq \E(L)^{\frac{1}{2}}\ln(n)^\frac{1+\kappa}{2}) \\
		&\geq \Prob(M \leq E(M)-\E(M)^{\frac{1}{2}}\ln(n)^\frac{1+\kappa}{2}) + \Prob(L\geq \E(L)+\E(L)^{\frac{1}{2}}\ln(n)^\frac{1+\kappa}{2})\\
		&\geq \Prob(M\leq a) + \Prob(L\geq a)
	\end{align*}
	if $a\in [\E(L)+\E(L)^{\frac{1}{2}}\ln(n)^\frac{1+\kappa}{2}, \E(M)-\E(M)^{\frac{1}{2}}\ln(n)^\frac{1+\kappa}{2}]=:I$. \\
	Now, since $\E(M) - \E(\tilde{L})\geq c\  d(z,\partial_I B_\circ(n))^\alpha$ and $\E(M)\leq c n^\alpha$, the above interval $I$ is nonempty if $d(z,\partial_I B_\circ(n))^\alpha \geq c n^{\frac{\alpha}{2}} \ln(n)^\frac{1+\kappa}{2}$.
	To finish the proof one uses the upper bound of $p_n$ with Borel-Cantelli. $\square$
	
	The bottleneck of the proof is obviously the very large upper bound of $\E(M) \leq c n^\alpha$. Sadly, this poor bound can not be improved without considering the specific position in the graph: Consider the boundary points of proper triangles $z$ at distance $2^m$ from $\circ$ which satisfy $\Prob_\circ(\tau_z < \tau(n)) \geq \frac{1}{2}$ for any $n> 2^m$ and therefore $\E(M)\geq c n^\alpha$.
	
	\section{The outer bound}
	The improved outer bound is a direct consequence from our inner bound: With very little adjustments one can deduce our bound from the results in \cite[Section 3.2]{idlaSG}. Since our result can be applied to any inner bound, and for the sake of completeness, we will quickly sketch this here.
	
	The idea of the proof is to consecutively stop the particles generating the cluster when leaving balls of growing radius. For this we will need a little more general notation than we introduced in Section 1.	
	For an already existing cluster $S\subseteq SG$ we define the IDLA cluster after starting a particle in vertex $x$ and stopping upon leaving the set $A$ as
	\begin{align*}
		\mathcal{I}(S;x\rightarrow A) := S\  \cup \big\{X\big(\min(\sigma_S,\sigma_A-1)\big)\big\},
	\end{align*}
	where $\sigma_A:= \inf \{t\geq 0 | X(t) \notin A \}$ and the paused particles as
	\begin{align*}
		\mathcal{P}(S;x\rightarrow A) := \begin{cases}
			X(\sigma_A) \text{ if } \sigma_A \leq \sigma_S \\
			\bot \text{ otherwise,}
		\end{cases}
	\end{align*}
	where $\bot$ connotes that the random walk was not paused and the particle has settled in the cluster. For starting consecutive particles from vertices $x_1,\ldots,x_n$ we write $\mathcal{I}(S;x_1,\ldots,x_n\rightarrow A)$ for the resulting cluster and $\mathcal{P}(S;x_1,\ldots,x_n\rightarrow A)$ for the sequence of paused particles.
	Due to the abelian property of the IDLA cluster (see \cite{growthAlgebra}), that is \break
	\begin{align*}
		\mathcal{I}\Big( \mathcal{I}(S;x_1,\ldots,x_n\rightarrow A); \mathcal{P}(S;x_1,\ldots,x_n\rightarrow A)\Big) \overset{\mathcal{D}}{=} \mathcal{I}(S;x_1,\ldots,x_n\rightarrow SG),
	\end{align*}
	it is possible to work on consecutively stopped clusters instead of the unstopped one.
	For the ease of notation we denote
	\begin{align*}
		&\mathcal{I}_n(x\rightarrow r) := \mathcal{I}(\emptyset; \underbrace{\circ,\ldots, \circ}_{n \text{ times}} \rightarrow B_\circ(r)), &\mathcal{P}_n(x\rightarrow r) := \mathcal{P}(\emptyset; \underbrace{\circ,\ldots, \circ}_{n \text{ times}} \rightarrow B_\circ(r)).
	\end{align*}
	
	The next lemma will be essential for the proof and claims that with high probability at least a proportion $\delta$ of started particles will settle before reaching the next radius at wich they will be stopped again.
	\begin{lemma}[{\cite[Lemma 3.7]{idlaSG}}]\label{outBoundHilfLemma}
		There are $\delta >0$ and $p<1$, such that for all $n$ large enough with $n^{\frac{1}{\alpha+1}}<k< n^{\alpha}$ we have for $x_1,\ldots,x_k\in B_\circ(n)$ and $S\subseteq B_\circ(n)$
		\begin{align*}
			\Prob\Big(\big| \mathcal{I}(S;x_1,\ldots,x_k \rightarrow B_\circ(n+k^\frac{1}{\alpha})\big) \backslash S \big| \leq \delta k \Big) \leq p^k.
		\end{align*}
	\end{lemma}

	With this we are now able to establish an outer bound depending on the already proved inner bound in Section \ref{sectionInnerBound}.	
	
	\begin{thm}[Outer bound in dependence of the inner bound]
		Suppose we already established an inner bound $D_I(n)$ on $SG$ such that 
		$$\sum_{n\in\N}\Prob \Big(B_\circ(n-D_I(n)) \nsubseteq \mathcal{I}_{b_n}(\circ\rightarrow n)\Big) < \infty $$
		
		Then for any $D \geq D_I^{1-\frac{1}{\alpha}} n^\frac{1}{\alpha}$ it holds with probability $1$ for $n$ large enough
		\begin{align*}
			\mathcal{I}(b_n) \subseteq B_\circ(n+D).
		\end{align*}  
	\end{thm}
	
\begin{proof}
	We will show that the probability of the event $\{\mathcal{I}(b_n) \nsubseteq B_\circ(n+D)\}$ is summable. The claim then follows directly from the Borel-Cantelli Lemma.
	
	For this, let
	\begin{align*}
		n_0 &:= n \text{,} &n_{j+1} &:= \begin{cases}
			n_j+k_j^{\frac{1}{\alpha}} \text{ if } k_j > n^{\frac{1}{\alpha+1}}_j \\
			\infty \text{ otherwise,}
		\end{cases}\\
		\mathcal{I}_0 &:= \mathcal{I}_{b_n}(0 \rightarrow n) \text{,} &\mathcal{I}_{j+1} &:= \mathcal{I}\big(\mathcal{I}_{j}; \mathcal{P}_j \rightarrow B_\circ(n_{j+1})\big), \\
		\mathcal{P}_0 &:= \mathcal{P}_{b_n}(0 \rightarrow n) \text{,} & \mathcal{P}_{j+1} &:= \mathcal{P}\big(\mathcal{I}_{j}; \mathcal{P}_{j} \rightarrow  B_\circ(n_{j+1})\big),\\
		k_0 &:= |\mathcal{P}_0| & k_{j+1} &:= |\mathcal{P}_{j+1}|
	\end{align*}
	be ascending clusters, whose union form the IDLA cluster. This iterative construction will stop, once we do not have enough stopped particles in $\mathcal{P}_j$. 
	
	Let $J:=\min\{j | k_j \leq n_j^{\frac{1}{\alpha+1}} \}$ be the time after which we let all particles in $\mathcal{P}_J$ evolve until settlement. Now for $j>J$ we have $\mathcal{I}_j = \mathcal{I}_{j+1}$ and by the abelian property of the IDLA cluster $\mathcal{I}_{J+1}$ and $\mathcal{I}(b_n)$ have the same distribution. By construction $\mathcal{I}_J\subseteq B_\circ(n_J)$ and since $k_J \leq n_J^{\frac{1}{\alpha+1}}$ we can deduce $\mathcal{I}_{J+1}\subseteq B_\circ(n_J+n_J^{\frac{1}{\alpha+1}})$. Therefore we have
	\begin{align*}
		\Prob\Big(\mathcal{I}(b_n) \nsubseteq B_\circ(n+D) \Big) &= \Prob\Big(\mathcal{I}_{J+1} \nsubseteq B_\circ(n+D) \Big) \\
		&\leq \Prob\Big(B_\circ(n_J + n_J^{\frac{1}{\alpha+1}} )\nsubseteq B_\circ(n+D) \Big) \\
		&= \Prob\Big(n_J+ n_J^{\frac{1}{\alpha+1}} > n+D \Big) \text{.}
	\end{align*}

	For $SG$ there is a nice annulus growth bound shown in \cite[Lemma 3.8]{idlaSG}: for all $\varepsilon>0$ it holds $b_n-b_{n(1-\varepsilon)} \leq 4 \varepsilon^{\alpha-1} b_n$ for $n$ large enough. Furthermore, by assumption and Borel-Cantelli we have with probability $1$ for $n$ large enough $B_\circ(n-D_I)\subseteq \mathcal{I}_{b_n}(0\rightarrow n)$. From this we can deduce $k_0 \leq b_n - |\mathcal{I}_{b_n}(0 \rightarrow n)| \leq b_n - b_{n-D_I} \leq c D_I^{\alpha-1} n < n^\alpha$ almost surely for $n$ large enough. 
	Therefore we have
	 \begin{align*} 
	 	n_J^\frac{1}{\alpha+1} < k_{J-1} < \ldots < k_0 < n^\alpha \text{ almost surely.}
	 \end{align*}
 	So the amount of paused particles $k_{j-1}$ satisfies $n_j^\frac{1}{\alpha+ 1} < k_{j-1} < n^\alpha < n_j^\alpha$ almost surely and we can apply Lemma \ref{outBoundHilfLemma}. Accordingly there are $\delta<1$ and $p<1$ such that 
 	\begin{align*}
 		\Prob\Big( \big| \mathcal{I}\big(S; \mathcal{P}_{j-1} \rightarrow B_\circ(n_{j-1}+l^\frac{1}{\alpha})\big)\backslash S \big| \leq (1-\delta) l \Big) \leq p^l.
 	\end{align*}
	Notice that $k_j - k_{j-1} = \big|\mathcal{I}_j\backslash\mathcal{I}_{j-1}\big|=\big| \mathcal{I}\big(\mathcal{I}_{j-1}; \mathcal{P}_{j-1} \rightarrow B_\circ(n_{j-1}+k_{j-1}^\frac{1}{\alpha})\big)\backslash \mathcal{I}_{j-1}\big|$ which are exactly number of settled particles in wave $j$. With the total law of probability we get
 	\begin{align} 
 		&\Prob\big(k_j \geq (1-\delta) k_{j-1} \big| n_{j-1}=i, k_{j-1} = l\big)\notag  \\
 		&= \sum_{S\subset B_\circ(i)} \Prob\Big( \big| \mathcal{I}\big(S; \mathcal{P}_{j-1} \rightarrow B_\circ(n_{j-1}+l^\frac{1}{\alpha})\big) \backslash S \big| \leq (1-\delta) l \Big| \mathcal{I}_{j-1} = S, n_{j-1}=i, k_{j-1} = l \Big) \cdot \notag\\[-3mm]
 		& \hspace{1.75cm}\Prob(\mathcal{I}_{j-1} = S \big| n_{j-1}=i, k_{j-1} = l) \notag\\
 		&\leq \sum_{\substack{S\subset B_\circ(i)}} p^l \Prob(\mathcal{I}_{j-1} = S \big| n_{j-1}=i, k_{j-1} = l) = p^l. \label{EquationStar}
 	\end{align}
 	With this at hand, by applying the total law of probability twice, we can deduce
 	\begin{align*}
 		&\Prob\big(k_j\geq (1-\delta)k_{j-1} \cap j \leq J \big) \phantom{\sum_i^k}\\
 		&\leq \sum_{i=n}^\infty \Prob(n_{j-1} = i ) \sum_{l = i^{1/(\alpha+1)}}^{n^\alpha} \Prob\big(k_j \geq (1-\delta) k_{j-1} | n_{j-1}=i, k_{j-1} = l\big)\cdot \\[-3.5mm]
 		& \hspace{7.5cm}\Prob(k_{j-1} = l \cap j \leq J | n_{j-1} = i ) \\
 		&\stackrel{(\ref{EquationStar})}{\leq} \sum_{i=n}^\infty \Prob(n_{j-1} = i ) \sum_{l = i^{1/(\alpha+1)}}^{n^\alpha} p^l \Prob(k_{j-1} = l \cap j \leq J | n_{j-1} = i ) \leq p^{n^{1/(\alpha+1)}}
 	\end{align*}
 	
 	and using this bound finally gives
 	\begin{align*}
 		\Prob\Big(\exists_{1\leq j \leq J} ~k_j \geq (1-\delta) k_{j-1} \Big) 
 		\leq \sum_{j=1}^{n^\alpha} \Prob\big(k_j\geq (1-\delta)k_{j-1} \cap j \leq J \big) \leq n^\alpha p^{n^{1/(\alpha+1)}}.
 	\end{align*}

 	For the complementary events we have 
 	\begin{align*} 
 		&\{ \forall_{1\leq j \leq l} ~k_j < (1-\delta) k_{j-1}\} \subseteq \{ k_l < (1-\delta)^l k_0 \}  \\
 		\Leftrightarrow &\{ \exists_{1\leq j \leq l} ~k_j \geq (1-\delta) k_{j-1}\} \supseteq \{ k_l \geq (1-\delta)^l k_0 \}
 	\end{align*}
	and therefore the according probabilities satisfy
 	\begin{align*}
 		\Prob\big(\exists_{1\leq l \leq J} ~k_l \geq (1-\delta)^l k_0 \big) 
 		= \Prob\big(\exists_{1\leq j \leq J} ~k_j \geq (1-\delta) k_{j-1}\big)
 		\leq  n^\alpha p^{n^{1/(\alpha+1)}}.
 	\end{align*}
 	On the complementary event $\{\forall_{1\leq j \leq J} ~k_j < (1-\delta)^j k_0) \}$ the following inequality holds: 
 	\begin{align*}
 		n_J = n + \sum_{j=0}^{J-1} k_j^\frac{1}{\alpha} < n + k_0^\frac{1}{\alpha} \sum_{j=0}^{J-1} \big((1-\delta)^\frac{1}{\alpha}\big)^j < n+ k_0^\frac{1}{\alpha}  \frac{1}{1-(1-\delta)^\frac{1}{\alpha}} \text{.}
 	\end{align*}
 	If additionally $n_J + n_J^{1/(\alpha+1)} > n +D$ holds, there is $c'= \Big(1+ \frac{1}{1-(1-\delta)^\frac{1}{\alpha}}\Big)^{1/(\alpha+1)}$, such that for $n$ large enough
 	\begin{align*}
 		&n+ k_0^\frac{1}{\alpha}  \frac{1}{1-(1-\delta)^\frac{1}{\alpha}} + c' n^{1/(\alpha+1)} \\
 		&> n+ k_0^\frac{1}{\alpha}  \frac{1}{1-(1-\delta)^\frac{1}{\alpha}} +  \Big(n+ k_0^\frac{1}{\alpha}  \frac{1}{1-(1-\delta)^\frac{1}{\alpha}}\Big)^{1/(\alpha+1)} \\
 		&> n_J +  n_J^{1/(\alpha+1)} > n +D
 	\end{align*}
 	From this we can easily deduce
 	\begin{align*}
 		&k_0^\frac{1}{\alpha} > (D- c' n^{1/(\alpha+1)})(1-(1-\delta)^\frac{1}{\alpha}) > c_1 D  \text{ for $n$ large enough}\\
 		&\Rightarrow k_0 > c_2 D^\alpha \text{.}
 	\end{align*}
 	By conditioning on the event $n_J < n+ k_0^\frac{1}{\alpha}  \frac{1}{1-(1-\delta)^\frac{1}{\alpha}}$ we obtain
 	\begin{align*}
 		&\Prob(\mathcal{I}(b_n) \nsubseteq B_\circ(n+D) ) \leq \Prob\Big(n_J+ n_J^{\frac{1}{\alpha+1}} > n+D \Big) \\
 		&= \Prob\Big(n_J+ n_J^{\frac{1}{\alpha+1}} > n+D \Big| n_J < n+ k_0^\frac{1}{\alpha}  \frac{1}{1-(1-\delta)^\frac{1}{\alpha}} \Big) \Prob\Big( n_J < n+ k_0^\frac{1}{\alpha}  \frac{1}{1-(1-\delta)^\frac{1}{\alpha}}\Big) \\
 		&+ \Prob\Big(n_J+ n_J^{\frac{1}{\alpha+1}} > n+D \Big| n_J \geq n+ k_0^\frac{1}{\alpha}  \frac{1}{1-(1-\delta)^\frac{1}{\alpha}} \Big) \Prob\Big( n_J \geq n+ k_0^\frac{1}{\alpha}  \frac{1}{1-(1-\delta)^\frac{1}{\alpha}}\Big) \\
 		&\leq \Prob\Big( k_0 > c D^\alpha\Big)+n^\alpha p^{n^{1/(\alpha+1)}}.
 	\end{align*}
 	Since by assumption $\Prob\big( k_0 > c D_I^{\alpha-1} n \big) \leq \Prob\big( B_\circ(n-D_I) \nsubseteq \mathcal{I}(b_n)\big) $ is summable, the above probability for any $D \geq D_I^{1-\frac{1}{\alpha}} n^\frac{1}{\alpha}$ is summable as well. Now applying Borel-Cantelli yields
 	\begin{align*}
 		\mathcal{I}(b_n) \subseteq B_\circ\big(n+D\big) \text{ almost surely for $n$ large enough.}
 	\end{align*}
 \end{proof}
 Note that in Section \ref{sectionInnerBound} we have actually shown an inner bound for the stopped cluster $\mathcal{I}_{b_n}(\circ \rightarrow n)$ as is needed here. This is because the considered $M,L$ only depend on the behavior of the random walks before leaving the ball of radius $n$. Therefore \break $\{ z\in \mathcal{I}_{b_n}(\circ \rightarrow n)\} \subseteq \{M=L\}$ holds for any $z\in B_\circ(n)$ as well.
 Applying the theorem to the established inner bound $D_I(n) = c n^\frac{1}{2} \ln(n)^{\frac{1+\kappa}{2\alpha}}$ yields the desired outer bound $$D= D_I^{1-\frac{1}{\alpha}} n^\frac{1}{\alpha} = c\ n^{\frac{1}{2}+\frac{1}{2\alpha}}\ln(n)^{\frac{1+\kappa}{2\alpha}(1-\frac{1}{\alpha})} \text.$$
 
 \section{Conclusion}
 	
	Firstly, to further improve the bounds derived here, it is necessary to consider different techniques. For the inner bound we used the difference between the expected visits in vertices and visits after settlement of the generating random walks. Our lower bound for this difference is already pretty sharp, since to odometer of the divisible sandpile gives the exact solution for this and the bound for the odometer function itself is also sharp for some special vertices. Considering the outer bound, our technique rests on stopping on the boundary of growing balls and the let paused particles develop on the annulus to the next bigger ball. In some sense we consider the balls on whose boundary the particles are stopped to be to be already completely settled in the subsequent process (since in Lemma \ref{outBoundHilfLemma} there is no further assumption on the sets $S\subseteq B_\circ(n)$). Now suppose the ball is a proper triangle of size $n=2^m$ and suppose we have filled the ball perfectly up to $b_{n-D_I}$, where $D_I = 2^k$, then the remaining particles amount to $$b_n-b_{n-D_I} = \frac{2^{m-k}(3^{k+1}+3)}{3^{k+1}+2} b_n \geq c D_I^{\alpha-1} n.$$ The radius to cover all these particles would be at least $D \geq c D_I^{1-\frac{1}{\alpha}} n^\frac{1}{\alpha}$. Therefore this result is also optimal in the sense, that for better bounds one would need to consider different techniques in order to proof these. 
		 
	Secondly, this approach can also be used for the Vicsek graph, whose sandpile cluster is also the exact ball. Here the difficulty is mainly in the additional technicalities and the exact analysis of a normalized odometer function, which is needed since the graph is not regular anymore. We expect this to also generalize to other nested fractals. However, one may need to consider another metric than the graph metric to describe the shape of the cluster, since the harmonic measure on the usual graph metric balls is not uniform anymore. Take a modified Sierpinski gasket graph where instead of three copies in each stage of the construction we take nine copies of previous stage triangles. Here the cluster will grow much faster on the outer middle triangle than on the other ones. Figure \ref{fig:SG3Sim} shows a simulation on the 6th generation of such a modified Sierpinski gasket graph, where the fluctuation is of the size of third generation triangle.
	
	\begin{figure}[h]
	\centering
	
	\includegraphics[width=\linewidth]{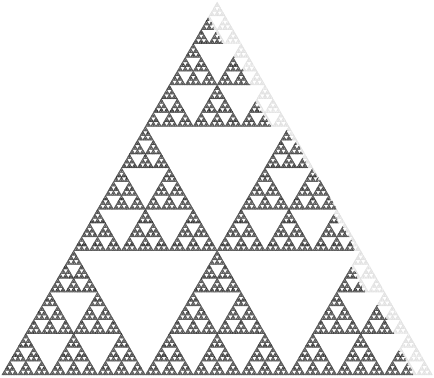}
	\caption{Simulation on the modified Sierpinski gasket graph with nine copies}\label{fig:SG3Sim}
	\end{figure}
	
 \newpage

(N.Heizmann) \textsc{ Technische Universität Chemnitz, Fakult\"at f\"ur Mathematik, Reichenhainer Stra\ss e 41, 09126 Chemnitz, Germany} \newline \textit{Email adress:} \textbf{nico.heizmann@math.tu-chemnitz.de}
\end{document}

%% file: tikzSG.tex
	
\begin{tikzpicture}[scale=0.3]
	\node[below] (0,0) {$\circ$};
  	\draw[] (0,0) -- (1,0) -- (0.5,0.5*1.7320508076) -- cycle;
    \begin{scope}[shift={(1,0)}]
      \draw[] (0,0) -- (1,0) -- (0.5,0.5*1.7320508076) -- cycle;
    \end{scope}
    \begin{scope}[shift={(1/2,0.5*1.7320508076)}]
      \draw[] (0,0) -- (1,0) -- (0.5,0.5*1.7320508076) -- cycle;
    \end{scope}
    \begin{scope}[shift={(2,0)}]
    	\draw[] (0,0) -- (1,0) -- (0.5,0.5*1.7320508076) -- cycle;
		\begin{scope}[shift={(1,0)}]
		  \draw[] (0,0) -- (1,0) -- (0.5,0.5*1.7320508076) -- cycle;
		\end{scope}
		\begin{scope}[shift={(1/2,0.5*1.7320508076)}]
		  \draw[] (0,0) -- (1,0) -- (0.5,0.5*1.7320508076) -- cycle;
		\end{scope}	
    \end{scope}
    \begin{scope}[shift={(1,1.7320508076)}]
    	\draw[] (0,0) -- (1,0) -- (0.5,0.5*1.7320508076) -- cycle;
		\begin{scope}[shift={(1,0)}]
		  \draw[] (0,0) -- (1,0) -- (0.5,0.5*1.7320508076) -- cycle;
		\end{scope}
		\begin{scope}[shift={(1/2,0.5*1.7320508076)}]
		  \draw[] (0,0) -- (1,0) -- (0.5,0.5*1.7320508076) -- cycle;
		\end{scope}	
    \end{scope}
	\begin{scope}[shift={(4,0)}]
		\draw[] (0,0) -- (1,0) -- (0.5,0.5*1.7320508076) -- cycle;
	    \begin{scope}[shift={(1,0)}]
	      \draw[] (0,0) -- (1,0) -- (0.5,0.5*1.7320508076) -- cycle;
	    \end{scope}
	    \begin{scope}[shift={(1/2,0.5*1.7320508076)}]
	      \draw[] (0,0) -- (1,0) -- (0.5,0.5*1.7320508076) -- cycle;
	    \end{scope}
	    \begin{scope}[shift={(2,0)}]
	    	\draw[] (0,0) -- (1,0) -- (0.5,0.5*1.7320508076) -- cycle;
			\begin{scope}[shift={(1,0)}]
			  \draw[] (0,0) -- (1,0) -- (0.5,0.5*1.7320508076) -- cycle;
			\end{scope}
			\begin{scope}[shift={(1/2,0.5*1.7320508076)}]
			  \draw[] (0,0) -- (1,0) -- (0.5,0.5*1.7320508076) -- cycle;
			\end{scope}	
	    \end{scope}
	    \begin{scope}[shift={(1,1.7320508076)}]
	    	\draw[] (0,0) -- (1,0) -- (0.5,0.5*1.7320508076) -- cycle;
			\begin{scope}[shift={(1,0)}]
			  \draw[] (0,0) -- (1,0) -- (0.5,0.5*1.7320508076) -- cycle;
			\end{scope}
			\begin{scope}[shift={(1/2,0.5*1.7320508076)}]
			  \draw[] (0,0) -- (1,0) -- (0.5,0.5*1.7320508076) -- cycle;
			\end{scope}	
	    \end{scope}	
	\end{scope}
	\begin{scope}[shift={(2,2*1.7320508076)}]
		\draw[] (0,0) -- (1,0) -- (0.5,0.5*1.7320508076) -- cycle;
	    \begin{scope}[shift={(1,0)}]
	      \draw[] (0,0) -- (1,0) -- (0.5,0.5*1.7320508076) -- cycle;
	    \end{scope}
	    \begin{scope}[shift={(1/2,0.5*1.7320508076)}]
	      \draw[] (0,0) -- (1,0) -- (0.5,0.5*1.7320508076) -- cycle;
	    \end{scope}
	    \begin{scope}[shift={(2,0)}]
	    	\draw[] (0,0) -- (1,0) -- (0.5,0.5*1.7320508076) -- cycle;
			\begin{scope}[shift={(1,0)}]
			  \draw[] (0,0) -- (1,0) -- (0.5,0.5*1.7320508076) -- cycle;
			\end{scope}
			\begin{scope}[shift={(1/2,0.5*1.7320508076)}]
			  \draw[] (0,0) -- (1,0) -- (0.5,0.5*1.7320508076) -- cycle;
			\end{scope}	
	    \end{scope}
	    \begin{scope}[shift={(1,1.7320508076)}]
	    	\draw[] (0,0) -- (1,0) -- (0.5,0.5*1.7320508076) -- cycle;
			\begin{scope}[shift={(1,0)}]
			  \draw[] (0,0) -- (1,0) -- (0.5,0.5*1.7320508076) -- cycle;
			\end{scope}
			\begin{scope}[shift={(1/2,0.5*1.7320508076)}]
			  \draw[] (0,0) -- (1,0) -- (0.5,0.5*1.7320508076) -- cycle;
			\end{scope}	
	    \end{scope}	
	\end{scope}
	\begin{scope}[shift={(8,0)}]
		\draw[] (0,0) -- (1,0) -- (0.5,0.5*1.7320508076) -- cycle;
	    \begin{scope}[shift={(1,0)}]
	      \draw[] (0,0) -- (1,0) -- (0.5,0.5*1.7320508076) -- cycle;
	    \end{scope}
	    \begin{scope}[shift={(1/2,0.5*1.7320508076)}]
	      \draw[] (0,0) -- (1,0) -- (0.5,0.5*1.7320508076) -- cycle;
	    \end{scope}
	    \begin{scope}[shift={(2,0)}]
	    	\draw[] (0,0) -- (1,0) -- (0.5,0.5*1.7320508076) -- cycle;
			\begin{scope}[shift={(1,0)}]
			  \draw[] (0,0) -- (1,0) -- (0.5,0.5*1.7320508076) -- cycle;
			\end{scope}
			\begin{scope}[shift={(1/2,0.5*1.7320508076)}]
			  \draw[] (0,0) -- (1,0) -- (0.5,0.5*1.7320508076) -- cycle;
			\end{scope}	
	    \end{scope}
	    \begin{scope}[shift={(1,1.7320508076)}]
	    	\draw[] (0,0) -- (1,0) -- (0.5,0.5*1.7320508076) -- cycle;
			\begin{scope}[shift={(1,0)}]
			  \draw[] (0,0) -- (1,0) -- (0.5,0.5*1.7320508076) -- cycle;
			\end{scope}
			\begin{scope}[shift={(1/2,0.5*1.7320508076)}]
			  \draw[] (0,0) -- (1,0) -- (0.5,0.5*1.7320508076) -- cycle;
			\end{scope}	
	    \end{scope}
		\begin{scope}[shift={(4,0)}]
			\draw[] (0,0) -- (1,0) -- (0.5,0.5*1.7320508076) -- cycle;
		    \begin{scope}[shift={(1,0)}]
		      \draw[] (0,0) -- (1,0) -- (0.5,0.5*1.7320508076) -- cycle;
		    \end{scope}
		    \begin{scope}[shift={(1/2,0.5*1.7320508076)}]
		      \draw[] (0,0) -- (1,0) -- (0.5,0.5*1.7320508076) -- cycle;
		    \end{scope}
		    \begin{scope}[shift={(2,0)}]
		    	\draw[] (0,0) -- (1,0) -- (0.5,0.5*1.7320508076) -- cycle;
				\begin{scope}[shift={(1,0)}]
				  \draw[] (0,0) -- (1,0) -- (0.5,0.5*1.7320508076) -- cycle;
				\end{scope}
				\begin{scope}[shift={(1/2,0.5*1.7320508076)}]
				  \draw[] (0,0) -- (1,0) -- (0.5,0.5*1.7320508076) -- cycle;
				\end{scope}	
		    \end{scope}
		    \begin{scope}[shift={(1,1.7320508076)}]
		    	\draw[] (0,0) -- (1,0) -- (0.5,0.5*1.7320508076) -- cycle;
				\begin{scope}[shift={(1,0)}]
				  \draw[] (0,0) -- (1,0) -- (0.5,0.5*1.7320508076) -- cycle;
				\end{scope}
				\begin{scope}[shift={(1/2,0.5*1.7320508076)}]
				  \draw[] (0,0) -- (1,0) -- (0.5,0.5*1.7320508076) -- cycle;
				\end{scope}	
		    \end{scope}	
		\end{scope}
		\begin{scope}[shift={(2,2*1.7320508076)}]
			\draw[] (0,0) -- (1,0) -- (0.5,0.5*1.7320508076) -- cycle;
		    \begin{scope}[shift={(1,0)}]
		      \draw[] (0,0) -- (1,0) -- (0.5,0.5*1.7320508076) -- cycle;
		    \end{scope}
		    \begin{scope}[shift={(1/2,0.5*1.7320508076)}]
		      \draw[] (0,0) -- (1,0) -- (0.5,0.5*1.7320508076) -- cycle;
		    \end{scope}
		    \begin{scope}[shift={(2,0)}]
		    	\draw[] (0,0) -- (1,0) -- (0.5,0.5*1.7320508076) -- cycle;
				\begin{scope}[shift={(1,0)}]
				  \draw[] (0,0) -- (1,0) -- (0.5,0.5*1.7320508076) -- cycle;
				\end{scope}
				\begin{scope}[shift={(1/2,0.5*1.7320508076)}]
				  \draw[] (0,0) -- (1,0) -- (0.5,0.5*1.7320508076) -- cycle;
				\end{scope}	
		    \end{scope}
		    \begin{scope}[shift={(1,1.7320508076)}]
		    	\draw[] (0,0) -- (1,0) -- (0.5,0.5*1.7320508076) -- cycle;
				\begin{scope}[shift={(1,0)}]
				  \draw[] (0,0) -- (1,0) -- (0.5,0.5*1.7320508076) -- cycle;
				\end{scope}
				\begin{scope}[shift={(1/2,0.5*1.7320508076)}]
				  \draw[] (0,0) -- (1,0) -- (0.5,0.5*1.7320508076) -- cycle;
				\end{scope}	
		    \end{scope}	
		\end{scope}	
	\end{scope}
	\begin{scope}[shift={(4,4*1.7320508076)}]
		\draw[] (0,0) -- (1,0) -- (0.5,0.5*1.7320508076) -- cycle;
	    \begin{scope}[shift={(1,0)}]
	      \draw[] (0,0) -- (1,0) -- (0.5,0.5*1.7320508076) -- cycle;
	    \end{scope}
	    \begin{scope}[shift={(1/2,0.5*1.7320508076)}]
	      \draw[] (0,0) -- (1,0) -- (0.5,0.5*1.7320508076) -- cycle;
	    \end{scope}
	    \begin{scope}[shift={(2,0)}]
	    	\draw[] (0,0) -- (1,0) -- (0.5,0.5*1.7320508076) -- cycle;
			\begin{scope}[shift={(1,0)}]
			  \draw[] (0,0) -- (1,0) -- (0.5,0.5*1.7320508076) -- cycle;
			\end{scope}
			\begin{scope}[shift={(1/2,0.5*1.7320508076)}]
			  \draw[] (0,0) -- (1,0) -- (0.5,0.5*1.7320508076) -- cycle;
			\end{scope}	
	    \end{scope}
	    \begin{scope}[shift={(1,1.7320508076)}]
	    	\draw[] (0,0) -- (1,0) -- (0.5,0.5*1.7320508076) -- cycle;
			\begin{scope}[shift={(1,0)}]
			  \draw[] (0,0) -- (1,0) -- (0.5,0.5*1.7320508076) -- cycle;
			\end{scope}
			\begin{scope}[shift={(1/2,0.5*1.7320508076)}]
			  \draw[] (0,0) -- (1,0) -- (0.5,0.5*1.7320508076) -- cycle;
			\end{scope}	
	    \end{scope}
		\begin{scope}[shift={(4,0)}]
			\draw[] (0,0) -- (1,0) -- (0.5,0.5*1.7320508076) -- cycle;
		    \begin{scope}[shift={(1,0)}]
		      \draw[] (0,0) -- (1,0) -- (0.5,0.5*1.7320508076) -- cycle;
		    \end{scope}
		    \begin{scope}[shift={(1/2,0.5*1.7320508076)}]
		      \draw[] (0,0) -- (1,0) -- (0.5,0.5*1.7320508076) -- cycle;
		    \end{scope}
		    \begin{scope}[shift={(2,0)}]
		    	\draw[] (0,0) -- (1,0) -- (0.5,0.5*1.7320508076) -- cycle;
				\begin{scope}[shift={(1,0)}]
				  \draw[] (0,0) -- (1,0) -- (0.5,0.5*1.7320508076) -- cycle;
				\end{scope}
				\begin{scope}[shift={(1/2,0.5*1.7320508076)}]
				  \draw[] (0,0) -- (1,0) -- (0.5,0.5*1.7320508076) -- cycle;
				\end{scope}	
		    \end{scope}
		    \begin{scope}[shift={(1,1.7320508076)}]
		    	\draw[] (0,0) -- (1,0) -- (0.5,0.5*1.7320508076) -- cycle;
				\begin{scope}[shift={(1,0)}]
				  \draw[] (0,0) -- (1,0) -- (0.5,0.5*1.7320508076) -- cycle;
				\end{scope}
				\begin{scope}[shift={(1/2,0.5*1.7320508076)}]
				  \draw[] (0,0) -- (1,0) -- (0.5,0.5*1.7320508076) -- cycle;
				\end{scope}	
		    \end{scope}	
		\end{scope}
		\begin{scope}[shift={(2,2*1.7320508076)}]
			\draw[] (0,0) -- (1,0) -- (0.5,0.5*1.7320508076) -- cycle;
		    \begin{scope}[shift={(1,0)}]
		      \draw[] (0,0) -- (1,0) -- (0.5,0.5*1.7320508076) -- cycle;
		    \end{scope}
		    \begin{scope}[shift={(1/2,0.5*1.7320508076)}]
		      \draw[] (0,0) -- (1,0) -- (0.5,0.5*1.7320508076) -- cycle;
		    \end{scope}
		    \begin{scope}[shift={(2,0)}]
		    	\draw[] (0,0) -- (1,0) -- (0.5,0.5*1.7320508076) -- cycle;
				\begin{scope}[shift={(1,0)}]
				  \draw[] (0,0) -- (1,0) -- (0.5,0.5*1.7320508076) -- cycle;
				\end{scope}
				\begin{scope}[shift={(1/2,0.5*1.7320508076)}]
				  \draw[] (0,0) -- (1,0) -- (0.5,0.5*1.7320508076) -- cycle;
				\end{scope}	
		    \end{scope}
		    \begin{scope}[shift={(1,1.7320508076)}]
		    	\draw[] (0,0) -- (1,0) -- (0.5,0.5*1.7320508076) -- cycle;
				\begin{scope}[shift={(1,0)}]
				  \draw[] (0,0) -- (1,0) -- (0.5,0.5*1.7320508076) -- cycle;
				\end{scope}
				\begin{scope}[shift={(1/2,0.5*1.7320508076)}]
				  \draw[] (0,0) -- (1,0) -- (0.5,0.5*1.7320508076) -- cycle;
				\end{scope}	
		    \end{scope}	
		\end{scope}	
	\end{scope}
	\begin{scope}[shift={(-16,0)}]
		\draw[] (0,0) -- (1,0) -- (0.5,0.5*1.7320508076) -- cycle;
	    \begin{scope}[shift={(1,0)}]
	      \draw[] (0,0) -- (1,0) -- (0.5,0.5*1.7320508076) -- cycle;
	    \end{scope}
	    \begin{scope}[shift={(1/2,0.5*1.7320508076)}]
	      \draw[] (0,0) -- (1,0) -- (0.5,0.5*1.7320508076) -- cycle;
	    \end{scope}
	    \begin{scope}[shift={(2,0)}]
	    	\draw[] (0,0) -- (1,0) -- (0.5,0.5*1.7320508076) -- cycle;
			\begin{scope}[shift={(1,0)}]
			  \draw[] (0,0) -- (1,0) -- (0.5,0.5*1.7320508076) -- cycle;
			\end{scope}
			\begin{scope}[shift={(1/2,0.5*1.7320508076)}]
			  \draw[] (0,0) -- (1,0) -- (0.5,0.5*1.7320508076) -- cycle;
			\end{scope}	
	    \end{scope}
	    \begin{scope}[shift={(1,1.7320508076)}]
	    	\draw[] (0,0) -- (1,0) -- (0.5,0.5*1.7320508076) -- cycle;
			\begin{scope}[shift={(1,0)}]
			  \draw[] (0,0) -- (1,0) -- (0.5,0.5*1.7320508076) -- cycle;
			\end{scope}
			\begin{scope}[shift={(1/2,0.5*1.7320508076)}]
			  \draw[] (0,0) -- (1,0) -- (0.5,0.5*1.7320508076) -- cycle;
			\end{scope}	
	    \end{scope}
		\begin{scope}[shift={(4,0)}]
			\draw[] (0,0) -- (1,0) -- (0.5,0.5*1.7320508076) -- cycle;
		    \begin{scope}[shift={(1,0)}]
		      \draw[] (0,0) -- (1,0) -- (0.5,0.5*1.7320508076) -- cycle;
		    \end{scope}
		    \begin{scope}[shift={(1/2,0.5*1.7320508076)}]
		      \draw[] (0,0) -- (1,0) -- (0.5,0.5*1.7320508076) -- cycle;
		    \end{scope}
		    \begin{scope}[shift={(2,0)}]
		    	\draw[] (0,0) -- (1,0) -- (0.5,0.5*1.7320508076) -- cycle;
				\begin{scope}[shift={(1,0)}]
				  \draw[] (0,0) -- (1,0) -- (0.5,0.5*1.7320508076) -- cycle;
				\end{scope}
				\begin{scope}[shift={(1/2,0.5*1.7320508076)}]
				  \draw[] (0,0) -- (1,0) -- (0.5,0.5*1.7320508076) -- cycle;
				\end{scope}	
		    \end{scope}
		    \begin{scope}[shift={(1,1.7320508076)}]
		    	\draw[] (0,0) -- (1,0) -- (0.5,0.5*1.7320508076) -- cycle;
				\begin{scope}[shift={(1,0)}]
				  \draw[] (0,0) -- (1,0) -- (0.5,0.5*1.7320508076) -- cycle;
				\end{scope}
				\begin{scope}[shift={(1/2,0.5*1.7320508076)}]
				  \draw[] (0,0) -- (1,0) -- (0.5,0.5*1.7320508076) -- cycle;
				\end{scope}	
		    \end{scope}	
		\end{scope}
		\begin{scope}[shift={(2,2*1.7320508076)}]
			\draw[] (0,0) -- (1,0) -- (0.5,0.5*1.7320508076) -- cycle;
		    \begin{scope}[shift={(1,0)}]
		      \draw[] (0,0) -- (1,0) -- (0.5,0.5*1.7320508076) -- cycle;
		    \end{scope}
		    \begin{scope}[shift={(1/2,0.5*1.7320508076)}]
		      \draw[] (0,0) -- (1,0) -- (0.5,0.5*1.7320508076) -- cycle;
		    \end{scope}
		    \begin{scope}[shift={(2,0)}]
		    	\draw[] (0,0) -- (1,0) -- (0.5,0.5*1.7320508076) -- cycle;
				\begin{scope}[shift={(1,0)}]
				  \draw[] (0,0) -- (1,0) -- (0.5,0.5*1.7320508076) -- cycle;
				\end{scope}
				\begin{scope}[shift={(1/2,0.5*1.7320508076)}]
				  \draw[] (0,0) -- (1,0) -- (0.5,0.5*1.7320508076) -- cycle;
				\end{scope}	
		    \end{scope}
		    \begin{scope}[shift={(1,1.7320508076)}]
		    	\draw[] (0,0) -- (1,0) -- (0.5,0.5*1.7320508076) -- cycle;
				\begin{scope}[shift={(1,0)}]
				  \draw[] (0,0) -- (1,0) -- (0.5,0.5*1.7320508076) -- cycle;
				\end{scope}
				\begin{scope}[shift={(1/2,0.5*1.7320508076)}]
				  \draw[] (0,0) -- (1,0) -- (0.5,0.5*1.7320508076) -- cycle;
				\end{scope}	
		    \end{scope}	
		\end{scope}
		\begin{scope}[shift={(8,0)}]
			\draw[] (0,0) -- (1,0) -- (0.5,0.5*1.7320508076) -- cycle;
		    \begin{scope}[shift={(1,0)}]
		      \draw[] (0,0) -- (1,0) -- (0.5,0.5*1.7320508076) -- cycle;
		    \end{scope}
		    \begin{scope}[shift={(1/2,0.5*1.7320508076)}]
		      \draw[] (0,0) -- (1,0) -- (0.5,0.5*1.7320508076) -- cycle;
		    \end{scope}
		    \begin{scope}[shift={(2,0)}]
		    	\draw[] (0,0) -- (1,0) -- (0.5,0.5*1.7320508076) -- cycle;
				\begin{scope}[shift={(1,0)}]
				  \draw[] (0,0) -- (1,0) -- (0.5,0.5*1.7320508076) -- cycle;
				\end{scope}
				\begin{scope}[shift={(1/2,0.5*1.7320508076)}]
				  \draw[] (0,0) -- (1,0) -- (0.5,0.5*1.7320508076) -- cycle;
				\end{scope}	
		    \end{scope}
		    \begin{scope}[shift={(1,1.7320508076)}]
		    	\draw[] (0,0) -- (1,0) -- (0.5,0.5*1.7320508076) -- cycle;
				\begin{scope}[shift={(1,0)}]
				  \draw[] (0,0) -- (1,0) -- (0.5,0.5*1.7320508076) -- cycle;
				\end{scope}
				\begin{scope}[shift={(1/2,0.5*1.7320508076)}]
				  \draw[] (0,0) -- (1,0) -- (0.5,0.5*1.7320508076) -- cycle;
				\end{scope}	
		    \end{scope}
			\begin{scope}[shift={(4,0)}]
				\draw[] (0,0) -- (1,0) -- (0.5,0.5*1.7320508076) -- cycle;
			    \begin{scope}[shift={(1,0)}]
			      \draw[] (0,0) -- (1,0) -- (0.5,0.5*1.7320508076) -- cycle;
			    \end{scope}
			    \begin{scope}[shift={(1/2,0.5*1.7320508076)}]
			      \draw[] (0,0) -- (1,0) -- (0.5,0.5*1.7320508076) -- cycle;
			    \end{scope}
			    \begin{scope}[shift={(2,0)}]
			    	\draw[] (0,0) -- (1,0) -- (0.5,0.5*1.7320508076) -- cycle;
					\begin{scope}[shift={(1,0)}]
					  \draw[] (0,0) -- (1,0) -- (0.5,0.5*1.7320508076) -- cycle;
					\end{scope}
					\begin{scope}[shift={(1/2,0.5*1.7320508076)}]
					  \draw[] (0,0) -- (1,0) -- (0.5,0.5*1.7320508076) -- cycle;
					\end{scope}	
			    \end{scope}
			    \begin{scope}[shift={(1,1.7320508076)}]
			    	\draw[] (0,0) -- (1,0) -- (0.5,0.5*1.7320508076) -- cycle;
					\begin{scope}[shift={(1,0)}]
					  \draw[] (0,0) -- (1,0) -- (0.5,0.5*1.7320508076) -- cycle;
					\end{scope}
					\begin{scope}[shift={(1/2,0.5*1.7320508076)}]
					  \draw[] (0,0) -- (1,0) -- (0.5,0.5*1.7320508076) -- cycle;
					\end{scope}	
			    \end{scope}	
			\end{scope}
			\begin{scope}[shift={(2,2*1.7320508076)}]
				\draw[] (0,0) -- (1,0) -- (0.5,0.5*1.7320508076) -- cycle;
			    \begin{scope}[shift={(1,0)}]
			      \draw[] (0,0) -- (1,0) -- (0.5,0.5*1.7320508076) -- cycle;
			    \end{scope}
			    \begin{scope}[shift={(1/2,0.5*1.7320508076)}]
			      \draw[] (0,0) -- (1,0) -- (0.5,0.5*1.7320508076) -- cycle;
			    \end{scope}
			    \begin{scope}[shift={(2,0)}]
			    	\draw[] (0,0) -- (1,0) -- (0.5,0.5*1.7320508076) -- cycle;
					\begin{scope}[shift={(1,0)}]
					  \draw[] (0,0) -- (1,0) -- (0.5,0.5*1.7320508076) -- cycle;
					\end{scope}
					\begin{scope}[shift={(1/2,0.5*1.7320508076)}]
					  \draw[] (0,0) -- (1,0) -- (0.5,0.5*1.7320508076) -- cycle;
					\end{scope}	
			    \end{scope}
			    \begin{scope}[shift={(1,1.7320508076)}]
			    	\draw[] (0,0) -- (1,0) -- (0.5,0.5*1.7320508076) -- cycle;
					\begin{scope}[shift={(1,0)}]
					  \draw[] (0,0) -- (1,0) -- (0.5,0.5*1.7320508076) -- cycle;
					\end{scope}
					\begin{scope}[shift={(1/2,0.5*1.7320508076)}]
					  \draw[] (0,0) -- (1,0) -- (0.5,0.5*1.7320508076) -- cycle;
					\end{scope}	
			    \end{scope}	
			\end{scope}	
		\end{scope}
		\begin{scope}[shift={(4,4*1.7320508076)}]
			\draw[] (0,0) -- (1,0) -- (0.5,0.5*1.7320508076) -- cycle;
		    \begin{scope}[shift={(1,0)}]
		      \draw[] (0,0) -- (1,0) -- (0.5,0.5*1.7320508076) -- cycle;
		    \end{scope}
		    \begin{scope}[shift={(1/2,0.5*1.7320508076)}]
		      \draw[] (0,0) -- (1,0) -- (0.5,0.5*1.7320508076) -- cycle;
		    \end{scope}
		    \begin{scope}[shift={(2,0)}]
		    	\draw[] (0,0) -- (1,0) -- (0.5,0.5*1.7320508076) -- cycle;
				\begin{scope}[shift={(1,0)}]
				  \draw[] (0,0) -- (1,0) -- (0.5,0.5*1.7320508076) -- cycle;
				\end{scope}
				\begin{scope}[shift={(1/2,0.5*1.7320508076)}]
				  \draw[] (0,0) -- (1,0) -- (0.5,0.5*1.7320508076) -- cycle;
				\end{scope}	
		    \end{scope}
		    \begin{scope}[shift={(1,1.7320508076)}]
		    	\draw[] (0,0) -- (1,0) -- (0.5,0.5*1.7320508076) -- cycle;
				\begin{scope}[shift={(1,0)}]
				  \draw[] (0,0) -- (1,0) -- (0.5,0.5*1.7320508076) -- cycle;
				\end{scope}
				\begin{scope}[shift={(1/2,0.5*1.7320508076)}]
				  \draw[] (0,0) -- (1,0) -- (0.5,0.5*1.7320508076) -- cycle;
				\end{scope}	
		    \end{scope}
			\begin{scope}[shift={(4,0)}]
				\draw[] (0,0) -- (1,0) -- (0.5,0.5*1.7320508076) -- cycle;
			    \begin{scope}[shift={(1,0)}]
			      \draw[] (0,0) -- (1,0) -- (0.5,0.5*1.7320508076) -- cycle;
			    \end{scope}
			    \begin{scope}[shift={(1/2,0.5*1.7320508076)}]
			      \draw[] (0,0) -- (1,0) -- (0.5,0.5*1.7320508076) -- cycle;
			    \end{scope}
			    \begin{scope}[shift={(2,0)}]
			    	\draw[] (0,0) -- (1,0) -- (0.5,0.5*1.7320508076) -- cycle;
					\begin{scope}[shift={(1,0)}]
					  \draw[] (0,0) -- (1,0) -- (0.5,0.5*1.7320508076) -- cycle;
					\end{scope}
					\begin{scope}[shift={(1/2,0.5*1.7320508076)}]
					  \draw[] (0,0) -- (1,0) -- (0.5,0.5*1.7320508076) -- cycle;
					\end{scope}	
			    \end{scope}
			    \begin{scope}[shift={(1,1.7320508076)}]
			    	\draw[] (0,0) -- (1,0) -- (0.5,0.5*1.7320508076) -- cycle;
					\begin{scope}[shift={(1,0)}]
					  \draw[] (0,0) -- (1,0) -- (0.5,0.5*1.7320508076) -- cycle;
					\end{scope}
					\begin{scope}[shift={(1/2,0.5*1.7320508076)}]
					  \draw[] (0,0) -- (1,0) -- (0.5,0.5*1.7320508076) -- cycle;
					\end{scope}	
			    \end{scope}	
			\end{scope}
			\begin{scope}[shift={(2,2*1.7320508076)}]
				\draw[] (0,0) -- (1,0) -- (0.5,0.5*1.7320508076) -- cycle;
			    \begin{scope}[shift={(1,0)}]
			      \draw[] (0,0) -- (1,0) -- (0.5,0.5*1.7320508076) -- cycle;
			    \end{scope}
			    \begin{scope}[shift={(1/2,0.5*1.7320508076)}]
			      \draw[] (0,0) -- (1,0) -- (0.5,0.5*1.7320508076) -- cycle;
			    \end{scope}
			    \begin{scope}[shift={(2,0)}]
			    	\draw[] (0,0) -- (1,0) -- (0.5,0.5*1.7320508076) -- cycle;
					\begin{scope}[shift={(1,0)}]
					  \draw[] (0,0) -- (1,0) -- (0.5,0.5*1.7320508076) -- cycle;
					\end{scope}
					\begin{scope}[shift={(1/2,0.5*1.7320508076)}]
					  \draw[] (0,0) -- (1,0) -- (0.5,0.5*1.7320508076) -- cycle;
					\end{scope}	
			    \end{scope}
			    \begin{scope}[shift={(1,1.7320508076)}]
			    	\draw[] (0,0) -- (1,0) -- (0.5,0.5*1.7320508076) -- cycle;
					\begin{scope}[shift={(1,0)}]
					  \draw[] (0,0) -- (1,0) -- (0.5,0.5*1.7320508076) -- cycle;
					\end{scope}
					\begin{scope}[shift={(1/2,0.5*1.7320508076)}]
					  \draw[] (0,0) -- (1,0) -- (0.5,0.5*1.7320508076) -- cycle;
					\end{scope}	
			    \end{scope}	
			\end{scope}	
		\end{scope}
	\end{scope}

	\begin{scope}[shift={(-24,0)}]
			\path[scope fading = south, fading angle = -120] (2.5,0-2) rectangle (8+2,4*1.7320508076-2);
			\draw[] (0,0) -- (1,0) -- (0.5,0.5*1.7320508076) -- cycle;
		    \begin{scope}[shift={(1,0)}]
		      \draw[] (0,0) -- (1,0) -- (0.5,0.5*1.7320508076) -- cycle;
		    \end{scope}
		    \begin{scope}[shift={(1/2,0.5*1.7320508076)}]
		      \draw[] (0,0) -- (1,0) -- (0.5,0.5*1.7320508076) -- cycle;
		    \end{scope}
		    \begin{scope}[shift={(2,0)}]
		    	\draw[] (0,0) -- (1,0) -- (0.5,0.5*1.7320508076) -- cycle;
				\begin{scope}[shift={(1,0)}]
				  \draw[] (0,0) -- (1,0) -- (0.5,0.5*1.7320508076) -- cycle;
				\end{scope}
				\begin{scope}[shift={(1/2,0.5*1.7320508076)}]
				  \draw[] (0,0) -- (1,0) -- (0.5,0.5*1.7320508076) -- cycle;
				\end{scope}	
		    \end{scope}
		    \begin{scope}[shift={(1,1.7320508076)}]
		    	\draw[] (0,0) -- (1,0) -- (0.5,0.5*1.7320508076) -- cycle;
				\begin{scope}[shift={(1,0)}]
				  \draw[] (0,0) -- (1,0) -- (0.5,0.5*1.7320508076) -- cycle;
				\end{scope}
				\begin{scope}[shift={(1/2,0.5*1.7320508076)}]
				  \draw[] (0,0) -- (1,0) -- (0.5,0.5*1.7320508076) -- cycle;
				\end{scope}	
		    \end{scope}
			\begin{scope}[shift={(4,0)}]
				\draw[] (0,0) -- (1,0) -- (0.5,0.5*1.7320508076) -- cycle;
			    \begin{scope}[shift={(1,0)}]
			      \draw[] (0,0) -- (1,0) -- (0.5,0.5*1.7320508076) -- cycle;
			    \end{scope}
			    \begin{scope}[shift={(1/2,0.5*1.7320508076)}]
			      \draw[] (0,0) -- (1,0) -- (0.5,0.5*1.7320508076) -- cycle;
			    \end{scope}
			    \begin{scope}[shift={(2,0)}]
			    	\draw[] (0,0) -- (1,0) -- (0.5,0.5*1.7320508076) -- cycle;
					\begin{scope}[shift={(1,0)}]
					  \draw[] (0,0) -- (1,0) -- (0.5,0.5*1.7320508076) -- cycle;
					\end{scope}
					\begin{scope}[shift={(1/2,0.5*1.7320508076)}]
					  \draw[] (0,0) -- (1,0) -- (0.5,0.5*1.7320508076) -- cycle;
					\end{scope}	
			    \end{scope}
			    \begin{scope}[shift={(1,1.7320508076)}]
			    	\draw[] (0,0) -- (1,0) -- (0.5,0.5*1.7320508076) -- cycle;
					\begin{scope}[shift={(1,0)}]
					  \draw[] (0,0) -- (1,0) -- (0.5,0.5*1.7320508076) -- cycle;
					\end{scope}
					\begin{scope}[shift={(1/2,0.5*1.7320508076)}]
					  \draw[] (0,0) -- (1,0) -- (0.5,0.5*1.7320508076) -- cycle;
					\end{scope}	
			    \end{scope}	
			\end{scope}
			\begin{scope}[shift={(2,2*1.7320508076)}]
				\draw[] (0,0) -- (1,0) -- (0.5,0.5*1.7320508076) -- cycle;
			    \begin{scope}[shift={(1,0)}]
			      \draw[] (0,0) -- (1,0) -- (0.5,0.5*1.7320508076) -- cycle;
			    \end{scope}
			    \begin{scope}[shift={(1/2,0.5*1.7320508076)}]
			      \draw[] (0,0) -- (1,0) -- (0.5,0.5*1.7320508076) -- cycle;
			    \end{scope}
			    \begin{scope}[shift={(2,0)}]
			    	\draw[] (0,0) -- (1,0) -- (0.5,0.5*1.7320508076) -- cycle;
					\begin{scope}[shift={(1,0)}]
					  \draw[] (0,0) -- (1,0) -- (0.5,0.5*1.7320508076) -- cycle;
					\end{scope}
					\begin{scope}[shift={(1/2,0.5*1.7320508076)}]
					  \draw[] (0,0) -- (1,0) -- (0.5,0.5*1.7320508076) -- cycle;
					\end{scope}	
			    \end{scope}
			    \begin{scope}[shift={(1,1.7320508076)}]
			    	\draw[] (0,0) -- (1,0) -- (0.5,0.5*1.7320508076) -- cycle;
					\begin{scope}[shift={(1,0)}]
					  \draw[] (0,0) -- (1,0) -- (0.5,0.5*1.7320508076) -- cycle;
					\end{scope}
					\begin{scope}[shift={(1/2,0.5*1.7320508076)}]
					  \draw[] (0,0) -- (1,0) -- (0.5,0.5*1.7320508076) -- cycle;
					\end{scope}	
			    \end{scope}	
			\end{scope}		
			
	\end{scope}
	
	\begin{scope}[shift={(-16,8*1.7320508076)}]
			\path[scope fading = south, fading angle = -120] (2.5,0-2) rectangle (8+2,4*1.7320508076-2);
			\draw[] (0,0) -- (1,0) -- (0.5,0.5*1.7320508076) -- cycle;
		    \begin{scope}[shift={(1,0)}]
		      \draw[] (0,0) -- (1,0) -- (0.5,0.5*1.7320508076) -- cycle;
		    \end{scope}
		    \begin{scope}[shift={(1/2,0.5*1.7320508076)}]
		      \draw[] (0,0) -- (1,0) -- (0.5,0.5*1.7320508076) -- cycle;
		    \end{scope}
		    \begin{scope}[shift={(2,0)}]
		    	\draw[] (0,0) -- (1,0) -- (0.5,0.5*1.7320508076) -- cycle;
				\begin{scope}[shift={(1,0)}]
				  \draw[] (0,0) -- (1,0) -- (0.5,0.5*1.7320508076) -- cycle;
				\end{scope}
				\begin{scope}[shift={(1/2,0.5*1.7320508076)}]
				  \draw[] (0,0) -- (1,0) -- (0.5,0.5*1.7320508076) -- cycle;
				\end{scope}	
		    \end{scope}
		    \begin{scope}[shift={(1,1.7320508076)}]
		    	\draw[] (0,0) -- (1,0) -- (0.5,0.5*1.7320508076) -- cycle;
				\begin{scope}[shift={(1,0)}]
				  \draw[] (0,0) -- (1,0) -- (0.5,0.5*1.7320508076) -- cycle;
				\end{scope}
				\begin{scope}[shift={(1/2,0.5*1.7320508076)}]
				  \draw[] (0,0) -- (1,0) -- (0.5,0.5*1.7320508076) -- cycle;
				\end{scope}	
		    \end{scope}
			\begin{scope}[shift={(4,0)}]
				\draw[] (0,0) -- (1,0) -- (0.5,0.5*1.7320508076) -- cycle;
			    \begin{scope}[shift={(1,0)}]
			      \draw[] (0,0) -- (1,0) -- (0.5,0.5*1.7320508076) -- cycle;
			    \end{scope}
			    \begin{scope}[shift={(1/2,0.5*1.7320508076)}]
			      \draw[] (0,0) -- (1,0) -- (0.5,0.5*1.7320508076) -- cycle;
			    \end{scope}
			    \begin{scope}[shift={(2,0)}]
			    	\draw[] (0,0) -- (1,0) -- (0.5,0.5*1.7320508076) -- cycle;
					\begin{scope}[shift={(1,0)}]
					  \draw[] (0,0) -- (1,0) -- (0.5,0.5*1.7320508076) -- cycle;
					\end{scope}
					\begin{scope}[shift={(1/2,0.5*1.7320508076)}]
					  \draw[] (0,0) -- (1,0) -- (0.5,0.5*1.7320508076) -- cycle;
					\end{scope}	
			    \end{scope}
			    \begin{scope}[shift={(1,1.7320508076)}]
			    	\draw[] (0,0) -- (1,0) -- (0.5,0.5*1.7320508076) -- cycle;
					\begin{scope}[shift={(1,0)}]
					  \draw[] (0,0) -- (1,0) -- (0.5,0.5*1.7320508076) -- cycle;
					\end{scope}
					\begin{scope}[shift={(1/2,0.5*1.7320508076)}]
					  \draw[] (0,0) -- (1,0) -- (0.5,0.5*1.7320508076) -- cycle;
					\end{scope}	
			    \end{scope}	
			\end{scope}
			\begin{scope}[shift={(2,2*1.7320508076)}]
				\draw[] (0,0) -- (1,0) -- (0.5,0.5*1.7320508076) -- cycle;
			    \begin{scope}[shift={(1,0)}]
			      \draw[] (0,0) -- (1,0) -- (0.5,0.5*1.7320508076) -- cycle;
			    \end{scope}
			    \begin{scope}[shift={(1/2,0.5*1.7320508076)}]
			      \draw[] (0,0) -- (1,0) -- (0.5,0.5*1.7320508076) -- cycle;
			    \end{scope}
			    \begin{scope}[shift={(2,0)}]
			    	\draw[] (0,0) -- (1,0) -- (0.5,0.5*1.7320508076) -- cycle;
					\begin{scope}[shift={(1,0)}]
					  \draw[] (0,0) -- (1,0) -- (0.5,0.5*1.7320508076) -- cycle;
					\end{scope}
					\begin{scope}[shift={(1/2,0.5*1.7320508076)}]
					  \draw[] (0,0) -- (1,0) -- (0.5,0.5*1.7320508076) -- cycle;
					\end{scope}	
			    \end{scope}
			    \begin{scope}[shift={(1,1.7320508076)}]
			    	\draw[] (0,0) -- (1,0) -- (0.5,0.5*1.7320508076) -- cycle;
					\begin{scope}[shift={(1,0)}]
					  \draw[] (0,0) -- (1,0) -- (0.5,0.5*1.7320508076) -- cycle;
					\end{scope}
					\begin{scope}[shift={(1/2,0.5*1.7320508076)}]
					  \draw[] (0,0) -- (1,0) -- (0.5,0.5*1.7320508076) -- cycle;
					\end{scope}	
			    \end{scope}	
			\end{scope}		
			
	\end{scope}

	\begin{scope}[shift = {(16,0)}]
	\path[scope fading = south, fading angle = 120] (-2.5,0-2) rectangle (8-2,4*1.7320508076-2);
			\draw[] (0,0) -- (1,0) -- (0.5,0.5*1.7320508076) -- cycle;
		    \begin{scope}[shift={(1,0)}]
		      \draw[] (0,0) -- (1,0) -- (0.5,0.5*1.7320508076) -- cycle;
		    \end{scope}
		    \begin{scope}[shift={(1/2,0.5*1.7320508076)}]
		      \draw[] (0,0) -- (1,0) -- (0.5,0.5*1.7320508076) -- cycle;
		    \end{scope}
		    \begin{scope}[shift={(2,0)}]
		    	\draw[] (0,0) -- (1,0) -- (0.5,0.5*1.7320508076) -- cycle;
				\begin{scope}[shift={(1,0)}]
				  \draw[] (0,0) -- (1,0) -- (0.5,0.5*1.7320508076) -- cycle;
				\end{scope}
				\begin{scope}[shift={(1/2,0.5*1.7320508076)}]
				  \draw[] (0,0) -- (1,0) -- (0.5,0.5*1.7320508076) -- cycle;
				\end{scope}	
		    \end{scope}
		    \begin{scope}[shift={(1,1.7320508076)}]
		    	\draw[] (0,0) -- (1,0) -- (0.5,0.5*1.7320508076) -- cycle;
				\begin{scope}[shift={(1,0)}]
				  \draw[] (0,0) -- (1,0) -- (0.5,0.5*1.7320508076) -- cycle;
				\end{scope}
				\begin{scope}[shift={(1/2,0.5*1.7320508076)}]
				  \draw[] (0,0) -- (1,0) -- (0.5,0.5*1.7320508076) -- cycle;
				\end{scope}	
		    \end{scope}
			\begin{scope}[shift={(4,0)}]
				\draw[] (0,0) -- (1,0) -- (0.5,0.5*1.7320508076) -- cycle;
			    \begin{scope}[shift={(1,0)}]
			      \draw[] (0,0) -- (1,0) -- (0.5,0.5*1.7320508076) -- cycle;
			    \end{scope}
			    \begin{scope}[shift={(1/2,0.5*1.7320508076)}]
			      \draw[] (0,0) -- (1,0) -- (0.5,0.5*1.7320508076) -- cycle;
			    \end{scope}
			    \begin{scope}[shift={(2,0)}]
			    	\draw[] (0,0) -- (1,0) -- (0.5,0.5*1.7320508076) -- cycle;
					\begin{scope}[shift={(1,0)}]
					  \draw[] (0,0) -- (1,0) -- (0.5,0.5*1.7320508076) -- cycle;
					\end{scope}
					\begin{scope}[shift={(1/2,0.5*1.7320508076)}]
					  \draw[] (0,0) -- (1,0) -- (0.5,0.5*1.7320508076) -- cycle;
					\end{scope}	
			    \end{scope}
			    \begin{scope}[shift={(1,1.7320508076)}]
			    	\draw[] (0,0) -- (1,0) -- (0.5,0.5*1.7320508076) -- cycle;
					\begin{scope}[shift={(1,0)}]
					  \draw[] (0,0) -- (1,0) -- (0.5,0.5*1.7320508076) -- cycle;
					\end{scope}
					\begin{scope}[shift={(1/2,0.5*1.7320508076)}]
					  \draw[] (0,0) -- (1,0) -- (0.5,0.5*1.7320508076) -- cycle;
					\end{scope}	
			    \end{scope}	
			\end{scope}
			\begin{scope}[shift={(2,2*1.7320508076)}]
				\draw[] (0,0) -- (1,0) -- (0.5,0.5*1.7320508076) -- cycle;
			    \begin{scope}[shift={(1,0)}]
			      \draw[] (0,0) -- (1,0) -- (0.5,0.5*1.7320508076) -- cycle;
			    \end{scope}
			    \begin{scope}[shift={(1/2,0.5*1.7320508076)}]
			      \draw[] (0,0) -- (1,0) -- (0.5,0.5*1.7320508076) -- cycle;
			    \end{scope}
			    \begin{scope}[shift={(2,0)}]
			    	\draw[] (0,0) -- (1,0) -- (0.5,0.5*1.7320508076) -- cycle;
					\begin{scope}[shift={(1,0)}]
					  \draw[] (0,0) -- (1,0) -- (0.5,0.5*1.7320508076) -- cycle;
					\end{scope}
					\begin{scope}[shift={(1/2,0.5*1.7320508076)}]
					  \draw[] (0,0) -- (1,0) -- (0.5,0.5*1.7320508076) -- cycle;
					\end{scope}	
			    \end{scope}
			    \begin{scope}[shift={(1,1.7320508076)}]
			    	\draw[] (0,0) -- (1,0) -- (0.5,0.5*1.7320508076) -- cycle;
					\begin{scope}[shift={(1,0)}]
					  \draw[] (0,0) -- (1,0) -- (0.5,0.5*1.7320508076) -- cycle;
					\end{scope}
					\begin{scope}[shift={(1/2,0.5*1.7320508076)}]
					  \draw[] (0,0) -- (1,0) -- (0.5,0.5*1.7320508076) -- cycle;
					\end{scope}	
			    \end{scope}	
			\end{scope}		
			
	\end{scope}
	
	\begin{scope}[shift = {(8,8*1.7320508076)}]
	\path[scope fading = south, fading angle = 120] (-2.5,0-2) rectangle (8-2,4*1.7320508076-2);
			\draw[] (0,0) -- (1,0) -- (0.5,0.5*1.7320508076) -- cycle;
		    \begin{scope}[shift={(1,0)}]
		      \draw[] (0,0) -- (1,0) -- (0.5,0.5*1.7320508076) -- cycle;
		    \end{scope}
		    \begin{scope}[shift={(1/2,0.5*1.7320508076)}]
		      \draw[] (0,0) -- (1,0) -- (0.5,0.5*1.7320508076) -- cycle;
		    \end{scope}
		    \begin{scope}[shift={(2,0)}]
		    	\draw[] (0,0) -- (1,0) -- (0.5,0.5*1.7320508076) -- cycle;
				\begin{scope}[shift={(1,0)}]
				  \draw[] (0,0) -- (1,0) -- (0.5,0.5*1.7320508076) -- cycle;
				\end{scope}
				\begin{scope}[shift={(1/2,0.5*1.7320508076)}]
				  \draw[] (0,0) -- (1,0) -- (0.5,0.5*1.7320508076) -- cycle;
				\end{scope}	
		    \end{scope}
		    \begin{scope}[shift={(1,1.7320508076)}]
		    	\draw[] (0,0) -- (1,0) -- (0.5,0.5*1.7320508076) -- cycle;
				\begin{scope}[shift={(1,0)}]
				  \draw[] (0,0) -- (1,0) -- (0.5,0.5*1.7320508076) -- cycle;
				\end{scope}
				\begin{scope}[shift={(1/2,0.5*1.7320508076)}]
				  \draw[] (0,0) -- (1,0) -- (0.5,0.5*1.7320508076) -- cycle;
				\end{scope}	
		    \end{scope}
			\begin{scope}[shift={(4,0)}]
				\draw[] (0,0) -- (1,0) -- (0.5,0.5*1.7320508076) -- cycle;
			    \begin{scope}[shift={(1,0)}]
			      \draw[] (0,0) -- (1,0) -- (0.5,0.5*1.7320508076) -- cycle;
			    \end{scope}
			    \begin{scope}[shift={(1/2,0.5*1.7320508076)}]
			      \draw[] (0,0) -- (1,0) -- (0.5,0.5*1.7320508076) -- cycle;
			    \end{scope}
			    \begin{scope}[shift={(2,0)}]
			    	\draw[] (0,0) -- (1,0) -- (0.5,0.5*1.7320508076) -- cycle;
					\begin{scope}[shift={(1,0)}]
					  \draw[] (0,0) -- (1,0) -- (0.5,0.5*1.7320508076) -- cycle;
					\end{scope}
					\begin{scope}[shift={(1/2,0.5*1.7320508076)}]
					  \draw[] (0,0) -- (1,0) -- (0.5,0.5*1.7320508076) -- cycle;
					\end{scope}	
			    \end{scope}
			    \begin{scope}[shift={(1,1.7320508076)}]
			    	\draw[] (0,0) -- (1,0) -- (0.5,0.5*1.7320508076) -- cycle;
					\begin{scope}[shift={(1,0)}]
					  \draw[] (0,0) -- (1,0) -- (0.5,0.5*1.7320508076) -- cycle;
					\end{scope}
					\begin{scope}[shift={(1/2,0.5*1.7320508076)}]
					  \draw[] (0,0) -- (1,0) -- (0.5,0.5*1.7320508076) -- cycle;
					\end{scope}	
			    \end{scope}	
			\end{scope}
			\begin{scope}[shift={(2,2*1.7320508076)}]
				\draw[] (0,0) -- (1,0) -- (0.5,0.5*1.7320508076) -- cycle;
			    \begin{scope}[shift={(1,0)}]
			      \draw[] (0,0) -- (1,0) -- (0.5,0.5*1.7320508076) -- cycle;
			    \end{scope}
			    \begin{scope}[shift={(1/2,0.5*1.7320508076)}]
			      \draw[] (0,0) -- (1,0) -- (0.5,0.5*1.7320508076) -- cycle;
			    \end{scope}
			    \begin{scope}[shift={(2,0)}]
			    	\draw[] (0,0) -- (1,0) -- (0.5,0.5*1.7320508076) -- cycle;
					\begin{scope}[shift={(1,0)}]
					  \draw[] (0,0) -- (1,0) -- (0.5,0.5*1.7320508076) -- cycle;
					\end{scope}
					\begin{scope}[shift={(1/2,0.5*1.7320508076)}]
					  \draw[] (0,0) -- (1,0) -- (0.5,0.5*1.7320508076) -- cycle;
					\end{scope}	
			    \end{scope}
			    \begin{scope}[shift={(1,1.7320508076)}]
			    	\draw[] (0,0) -- (1,0) -- (0.5,0.5*1.7320508076) -- cycle;
					\begin{scope}[shift={(1,0)}]
					  \draw[] (0,0) -- (1,0) -- (0.5,0.5*1.7320508076) -- cycle;
					\end{scope}
					\begin{scope}[shift={(1/2,0.5*1.7320508076)}]
					  \draw[] (0,0) -- (1,0) -- (0.5,0.5*1.7320508076) -- cycle;
					\end{scope}	
			    \end{scope}	
			\end{scope}		
			
	\end{scope}

	
	
	
\end{tikzpicture}